\documentclass[11pt]{article}

\usepackage{mathrsfs}
\usepackage{nicefrac}
\usepackage[T1]{fontenc}
\usepackage{latexsym,amssymb,amsmath,amsfonts,amsthm}
\usepackage{graphicx}
\usepackage{caption}
\usepackage[utf8]{inputenc}
\usepackage{authblk}
\usepackage{bm}
\usepackage{accents}
\usepackage{booktabs}
\usepackage[pdftex,colorlinks=true,citecolor=blue,linkcolor=blue]{hyperref}

\newtheorem{theorem}{Theorem}[section]
\newtheorem{lemma}{Lemma}[section]
\newtheorem{corollary}{Corollary}[section]
\newtheorem{remark}{Remark}[section]

\DeclareMathOperator*{\rank}{rank}
\DeclareMathOperator*{\diag}{diag}
\DeclareMathOperator*{\tr}{tr}
\DeclareMathOperator*{\Retr}{Re\,tr}

\begin{document}
	
\title{On the perturbation of the Moore--Penrose inverse of a matrix}
\author{
Xuefeng Xu%
\thanks{Department of Mathematics, Purdue University, West Lafayette, IN 47907, USA (\texttt{xuxuefeng@lsec.cc.ac.cn; xu1412@purdue.edu}).}
}
%
\date{\today}
\maketitle
	
\begin{abstract}

The Moore--Penrose inverse of a matrix has been extensively investigated and widely applied in many fields over the past decades. One reason for the interest is that the Moore--Penrose inverse can succinctly express some important geometric constructions in finite-dimensional spaces, such as the orthogonal projection onto a subspace and the linear least squares problem. In this paper, we establish new perturbation bounds for the Moore--Penrose inverse under the Frobenius norm, some of which are sharper than the existing ones.

\end{abstract}
	

\medskip


\section{Introduction}

Let $\mathbb{C}^{m\times n}$, $\mathbb{C}^{m\times n}_{r}$, and $\mathscr{U}_{n}$ be the set of all $m\times n$ complex matrices, the set of all $m\times n$ complex matrices of rank $r$, and the set of all $n\times n$ unitary matrices, respectively. For any $M\in\mathbb{C}^{m\times n}$, the symbols $M^{\ast}$, $\rank(M)$, $\|M\|_{\mathscr{U}}$, $\|M\|_{2}$, and $\|M\|_{F}$ denote the conjugate transpose, the rank, the unitarily invariant norm, the spectral norm, and the Frobenius norm of $M$, respectively.

The Moore--Penrose (MP) inverse of $M\in\mathbb{C}^{m\times n}$ is denoted by $M^{\dagger}$, which is defined as the unique matrix $X\in\mathbb{C}^{n\times m}$ satisfying the following equations~\cite{Penrose1955,Penrose1956}:
\begin{displaymath}
{\rm (i)} \ MXM=M, \quad {\rm (ii)} \ XMX=X, \quad {\rm (iii)} \ (MX)^{\ast}=MX, \quad {\rm (iv)} \ (XM)^{\ast}=XM.
\end{displaymath}
In particular, if $M$ is a square and nonsingular matrix, then $M^{\dagger}$ will coincide with the usual inverse $M^{-1}$. The MP inverse can concisely express some important geometric constructions in finite-dimensional spaces, such as the orthogonal projection onto a subspace and the linear least squares problem. More specifically, the orthogonal projection onto the column space of $A$ can be expressed as $P_{A}=AA^{\dagger}$; see~\cite{Xu2020} for the perturbation analysis of $P_{A}$. Recall that the linear least squares problem can be described as follows: Find $\mathbf{x}_{\star}\in\mathbb{C}^{n}$ such that
\begin{equation}\label{LSP}
\mathbf{x}_{\star}\in\mathop{\arg\min}_{\mathbf{x}\in\mathbb{C}^{n}}\|A\mathbf{x}-\mathbf{b}\|_{2},
\end{equation}
where $A\in\mathbb{C}^{m\times n}$ and $\mathbf{b}\in\mathbb{C}^{m}$. It is well known that the solutions of~\eqref{LSP} can be formulated as
\begin{displaymath}
\mathbf{x}_{\star}=A^{\dagger}\mathbf{b}+(I_{n}-A^{\dagger}A)\mathbf{z},
\end{displaymath}
where $I_{n}$ denotes the $n\times n$ identity matrix and $\mathbf{z}\in\mathbb{C}^{n}$ is an arbitrary vector. Furthermore, the minimum $2$-norm solution of~\eqref{LSP} is $\mathbf{x}_{\star}=A^{\dagger}\mathbf{b}$. The MP inverse has been widely applied in many fields such as matrix computation, algorithm analysis, statistics, and engineering; see, e.g.,~\cite{Ben2003,Drineas2011,Hoyle2011,Bodnar2016}. Over the past decades, many researchers have investigated the perturbation analysis of MP inverse, which can be found, e.g., in~\cite{Stewart1969,Wedin1973,Abdelmalek1974,Stewart1977,Sun2001,Meng2010,Cai2011,Xu2017,Li2018}.

Let $A\in\mathbb{C}^{m\times n}_{r}$, $B\in\mathbb{C}^{m\times n}_{s}$, and $E=B-A$. Wedin~\cite{Wedin1973} established the estimate (see also~\cite{Sun2001,Meng2010})
\begin{equation}\label{Wedin1}
\|B^{\dagger}-A^{\dagger}\|\leq\mu_{1}\max\big\{\|A^{\dagger}\|_{2}^{2},\|B^{\dagger}\|_{2}^{2}\big\}\|E\|.
\end{equation}
In particular, if $s=r$, then
\begin{equation}\label{Wedin2}
\|B^{\dagger}-A^{\dagger}\|\leq\mu_{2}\|A^{\dagger}\|_{2}\|B^{\dagger}\|_{2}\|E\|.
\end{equation}
The above parameters $\mu_{1}$ and $\mu_{2}$ are listed in Table~\ref{tab:Wedin}.
\medskip
\begin{table}[h!!]
\centering
\setlength{\tabcolsep}{2.6mm}{
\begin{tabular}{@{} lcccc @{}}
\toprule
\ $\|\cdot\|$ & $\mu_{1}$ & $\mu_{2}$ $\big(r<\min\{m,n\}\big)$ & $\mu_{2}$ $\Big(r=\min\limits_{m\neq n}\{m,n\}\Big)$ & $\mu_{2}$ $\big(r=m=n\big)$ \\
\midrule
\ $\|\cdot\|_{\mathscr{U}}$ & $3$ & $3$ & $2$ & $1$ \\
\ $\|\cdot\|_{2}$ & $\frac{1+\sqrt{5}}{2}$ & $\frac{1+\sqrt{5}}{2}$ & $\sqrt{2}$ & $1$ \\
\ $\|\cdot\|_{F}$ & $\sqrt{2}$ & $\sqrt{2}$ & $1$ & $1$ \\
\bottomrule
\end{tabular}}
\caption{\small The values of $\mu_{1}$ and $\mu_{2}$.}
\label{tab:Wedin}
\end{table}
In 2010, Meng and Zheng~\cite[Theorems~2.1 and 2.2]{Meng2010} improved the estimates~\eqref{Wedin1} and~\eqref{Wedin2} under the Frobenius norm. More specifically, they derived that
\begin{equation}\label{Meng1}
\|B^{\dagger}-A^{\dagger}\|_{F}\leq\max\big\{\|A^{\dagger}\|_{2}^{2},\|B^{\dagger}\|_{2}^{2}\big\}\|E\|_{F}.
\end{equation}
In particular, if $s=r$, then
\begin{equation}\label{Meng2}
\|B^{\dagger}-A^{\dagger}\|_{F}\leq\|A^{\dagger}\|_{2}\|B^{\dagger}\|_{2}\|E\|_{F}.
\end{equation}
Recently, Li et al.~\cite[Theorem 3.1]{Li2018} further refined the estimate~\eqref{Meng1}. They obtained that
\begin{equation}\label{Li1}
\|B^{\dagger}-A^{\dagger}\|_{F}^{2}\leq\max\big\{\|A^{\dagger}\|_{2}^{4},\|B^{\dagger}\|_{2}^{4}\big\}\|E\|_{F}^{2}-\frac{\|A^{\dagger}EB^{\dagger}\|_{F}^{2}+\|B^{\dagger}EA^{\dagger}\|_{F}^{2}}{2}\bigg(\max\bigg\{\frac{\|A^{\dagger}\|_{2}^{2}}{\|B^{\dagger}\|_{2}^{2}},\frac{\|B^{\dagger}\|_{2}^{2}}{\|A^{\dagger}\|_{2}^{2}}\bigg\}-1\bigg).
\end{equation}
If $A\in\mathbb{C}^{m\times n}_{n}$ ($m\geq n$) and $B=A+E\in\mathbb{C}^{m\times n}_{s}$, Li et al.~\cite[Theorem 3.2]{Li2018} also proved that
\begin{equation}\label{Li2.1}
\|B^{\dagger}-A^{\dagger}\|_{F}^{2}\leq\frac{\|A^{\dagger}\|_{2}^{2}\|B^{\dagger}\|_{2}^{2}}{\|A^{\dagger}\|_{2}^{2}+\|B^{\dagger}\|_{2}^{2}}\bigg(\|EA^{\dagger}\|_{F}^{2}+\|EB^{\dagger}\|_{F}^{2}+(n-s)\frac{\|A^{\dagger}\|_{2}^{2}}{\|B^{\dagger}\|_{2}^{2}}\bigg).
\end{equation}
In particular, if $s=n$, then
\begin{equation}\label{Li2.2}
\|B^{\dagger}-A^{\dagger}\|_{F}^{2}\leq\min\big\{\|B^{\dagger}\|_{2}^{2}\|EA^{\dagger}\|_{F}^{2},\|A^{\dagger}\|_{2}^{2}\|EB^{\dagger}\|_{F}^{2}\big\}.
\end{equation}

Although the estimate~\eqref{Li1} has sharpened~\eqref{Meng1}, the upper bound in~\eqref{Li1} is still too large in certain cases. We now give a simple example:
\begin{equation}\label{Ex0}
A=\begin{pmatrix}
1 & 0 \\
0 & 0
\end{pmatrix}, \quad B=\begin{pmatrix}
\frac{1}{1+2\tau} & 0 \\
0 & \tau
\end{pmatrix},
\end{equation}
where $0<\tau<\frac{1}{2}$. In this example, it holds that
$$
\|B^{\dagger}-A^{\dagger}\|_{F}^{2}=4\tau^{2}+\frac{1}{\tau^{2}}.
$$
Direct computation yields that the upper bound in~\eqref{Li1} is
\begin{displaymath}
4\tau^{2}+\frac{1}{\tau^{2}}+\frac{4}{\tau^{2}(1+2\tau)^{2}}-4=:u(\tau).
\end{displaymath}
It is easy to see that
\begin{displaymath}
u(\tau)-\|B^{\dagger}-A^{\dagger}\|_{F}^{2}=\frac{4}{\tau^{2}(1+2\tau)^{2}}-4,
\end{displaymath}
which will be very large if $0<\tau\ll\frac{1}{2}$. Moreover, if $\tau$ is sufficiently small, then
\begin{displaymath}
u(\tau)\simeq\frac{5}{\tau^{2}}\simeq 5\|B^{\dagger}-A^{\dagger}\|_{F}^{2}.
\end{displaymath}

Motivated by the above observation, we revisit the perturbation of MP inverse under the Frobenius norm. Some new upper bounds for $\|B^{\dagger}-A^{\dagger}\|_{F}^{2}$ are presented. Theoretical analysis shows that the new bounds are sharper than the existing ones.

The rest of this paper is organized as follows. In Section~\ref{sec:pre}, we introduce a trace inequality and several auxiliary results on $\|B^{\dagger}-A^{\dagger}\|_{F}^{2}$. In Section~\ref{sec:main}, we present some new upper bounds for $\|B^{\dagger}-A^{\dagger}\|_{F}^{2}$, and compare the new bounds with the existing ones theoretically.

\section{Preliminaries}

\label{sec:pre}

\setcounter{equation}{0}

Let $\{\sigma_{i}(M)\}_{i=1}^{t}$ and $\{\sigma_{i}(N)\}_{i=1}^{t}$ be the singular values of $M\in\mathbb{C}^{m\times n}$ and $N\in\mathbb{C}^{m\times n}$, respectively, where $t=\min\{m,n\}$. Assume that $\{\sigma_{i}(M)\}_{i=1}^{t}$ and $\{\sigma_{i}(N)\}_{i=1}^{t}$ are arranged in the same (increasing or decreasing) order. The celebrated von Neumann's trace inequality~\cite{Neumann1937} reads
\begin{displaymath}
\Retr(UMVN^{\ast})\leq\sum_{i=1}^{t}\sigma_{i}(M)\sigma_{i}(N),
\end{displaymath}
where $\Retr(\cdot)$ denotes the real part of the trace of a matrix, and both $U\in\mathscr{U}_{m}$ and $V\in\mathscr{U}_{n}$ are arbitrary. Indeed, the following more accurate characterization for $\Retr(UMVN^{\ast})$~\cite{Neumann1937} holds.

\begin{lemma}
Let $M\in\mathbb{C}^{m\times n}$, $N\in\mathbb{C}^{m\times n}$, and $t=\min\{m,n\}$. Let $\{\sigma_{i}(M)\}_{i=1}^{t}$ and $\{\sigma_{i}(N)\}_{i=1}^{t}$ be the singular values of $M$ and $N$, respectively, which are arranged in the same {\rm(}increasing or decreasing{\rm )} order. Then
\begin{equation}\label{von}
\max_{\substack{\,\, U\in\mathscr{U}_{m} \\ V\in\mathscr{U}_{n}}}\Retr(UMVN^{\ast})=\sum_{i=1}^{t}\sigma_{i}(M)\sigma_{i}(N).
\end{equation}
\end{lemma}

Based on the singular value decomposition (SVD) of a matrix, we can derive two characterizations of $\|B^{\dagger}-A^{\dagger}\|_{F}^{2}$ (see Lemma~\ref{identity}), which play a fundamental role in our analysis. Let $A\in\mathbb{C}^{m\times n}_{r}$ and $B\in\mathbb{C}^{m\times n}_{s}$ (\textit{throughout this paper, we only consider the nontrivial case that $r\geq 1$ and $s\geq 1$}) have the following SVDs:
\begin{subequations}
\begin{align}
&A=U\begin{pmatrix}
\Sigma_{1} & 0 \\
0 & 0
\end{pmatrix}V^{\ast}=U_{1}\Sigma_{1}V_{1}^{\ast},\label{SVD-A}\\
&B=\widetilde{U}\begin{pmatrix}
\widetilde{\Sigma}_{1} & 0 \\
0 & 0
\end{pmatrix}\widetilde{V}^{\ast}=\widetilde{U}_{1}\widetilde{\Sigma}_{1}\widetilde{V}_{1}^{\ast},\label{SVD-B}
\end{align}
\end{subequations}
where $U=(U_{1},U_{2})\in\mathscr{U}_{m}$, $V=(V_{1},V_{2})\in\mathscr{U}_{n}$, $\widetilde{U}=(\widetilde{U}_{1},\widetilde{U}_{2})\in\mathscr{U}_{m}$, $\widetilde{V}=(\widetilde{V}_{1},\widetilde{V}_{2})\in\mathscr{U}_{n}$, $U_{1}\in\mathbb{C}^{m\times r}$, $V_{1}\in\mathbb{C}^{n\times r}$, $\widetilde{U}_{1}\in\mathbb{C}^{m\times s}$, $\widetilde{V}_{1}\in\mathbb{C}^{n\times s}$, $\Sigma_{1}=\diag(\sigma_{1},\ldots,\sigma_{r})$, $\widetilde{\Sigma}_{1}=\diag(\widetilde{\sigma}_{1},\ldots,\widetilde{\sigma}_{s})$, $\sigma_{1}\geq\cdots\geq\sigma_{r}>0$, and $\widetilde{\sigma}_{1}\geq\cdots\geq\widetilde{\sigma}_{s}>0$. In view of~\eqref{SVD-A} and~\eqref{SVD-B}, the MP inverses $A^{\dagger}$ and $B^{\dagger}$ can be explicitly expressed as follows:
\begin{subequations}
\begin{align}
&A^{\dagger}=V\begin{pmatrix}
\Sigma_{1}^{-1} & 0 \\
0 & 0
\end{pmatrix}U^{\ast}=V_{1}\Sigma_{1}^{-1}U_{1}^{\ast},\label{MP-A}\\
&B^{\dagger}=\widetilde{V}\begin{pmatrix}
\widetilde{\Sigma}_{1}^{-1} & 0 \\
0 & 0
\end{pmatrix}\widetilde{U}^{\ast}=\widetilde{V}_{1}\widetilde{\Sigma}_{1}^{-1}\widetilde{U}_{1}^{\ast}.\label{MP-B}
\end{align}
\end{subequations}

Using~\eqref{SVD-A}, \eqref{SVD-B}, \eqref{MP-A}, and~\eqref{MP-B}, we can obtain the following identities.

\begin{lemma}\label{identity}
Let $A\in\mathbb{C}^{m\times n}_{r}$ and $B\in\mathbb{C}^{m\times n}_{s}$ have the SVDs~\eqref{SVD-A} and~\eqref{SVD-B}, respectively, and let $E=B-A$. Then
\begin{subequations}
\begin{align}
&\|B^{\dagger}-A^{\dagger}\|_{F}^{2}=\|\widetilde{\Sigma}_{1}^{-1}\widetilde{U}_{1}^{\ast}U_{2}\|_{F}^{2}+\|\widetilde{V}_{2}^{\ast}V_{1}\Sigma_{1}^{-1}\|_{F}^{2}+\|B^{\dagger}EA^{\dagger}\|_{F}^{2},\label{ide1.1}\\
&\|B^{\dagger}-A^{\dagger}\|_{F}^{2}=\|\widetilde{U}_{2}^{\ast}U_{1}\Sigma_{1}^{-1}\|_{F}^{2}+\|\widetilde{\Sigma}_{1}^{-1}\widetilde{V}_{1}^{\ast}V_{2}\|_{F}^{2}+\|A^{\dagger}EB^{\dagger}\|_{F}^{2}.\label{ide1.2}
\end{align}
\end{subequations}
\end{lemma}

\begin{proof}
By~\eqref{MP-A} and~\eqref{MP-B}, we have
\begin{align*}
&\widetilde{V}^{\ast}(B^{\dagger}-A^{\dagger})U=\begin{pmatrix}
\widetilde{\Sigma}_{1}^{-1}\widetilde{U}_{1}^{\ast}U_{1}-\widetilde{V}_{1}^{\ast}V_{1}\Sigma_{1}^{-1} & \widetilde{\Sigma}_{1}^{-1}\widetilde{U}_{1}^{\ast}U_{2} \\
-\widetilde{V}_{2}^{\ast}V_{1}\Sigma_{1}^{-1} & 0
\end{pmatrix},\\
&V^{\ast}(B^{\dagger}-A^{\dagger})\widetilde{U}=\begin{pmatrix}
V_{1}^{\ast}\widetilde{V}_{1}\widetilde{\Sigma}_{1}^{-1}-\Sigma_{1}^{-1}U_{1}^{\ast}\widetilde{U}_{1} & -\Sigma_{1}^{-1}U_{1}^{\ast}\widetilde{U}_{2} \\
V_{2}^{\ast}\widetilde{V}_{1}\widetilde{\Sigma}_{1}^{-1}  & 0
\end{pmatrix}.
\end{align*}
Then
\begin{subequations}
\begin{align}
&\|B^{\dagger}-A^{\dagger}\|_{F}^{2}=\|\widetilde{\Sigma}_{1}^{-1}\widetilde{U}_{1}^{\ast}U_{2}\|_{F}^{2}+\|\widetilde{V}_{2}^{\ast}V_{1}\Sigma_{1}^{-1}\|_{F}^{2}+\|\widetilde{\Sigma}_{1}^{-1}\widetilde{U}_{1}^{\ast}U_{1}-\widetilde{V}_{1}^{\ast}V_{1}\Sigma_{1}^{-1}\|_{F}^{2},\label{rela1.1}\\
&\|B^{\dagger}-A^{\dagger}\|_{F}^{2}=\|\Sigma_{1}^{-1}U_{1}^{\ast}\widetilde{U}_{2}\|_{F}^{2}+\|V_{2}^{\ast}\widetilde{V}_{1}\widetilde{\Sigma}_{1}^{-1}\|_{F}^{2}+\|V_{1}^{\ast}\widetilde{V}_{1}\widetilde{\Sigma}_{1}^{-1}-\Sigma_{1}^{-1}U_{1}^{\ast}\widetilde{U}_{1}\|_{F}^{2}.\label{rela1.2}
\end{align}
\end{subequations}
In addition, using~\eqref{SVD-A}, \eqref{SVD-B}, \eqref{MP-A}, and~\eqref{MP-B}, we get
\begin{align*}
&\widetilde{V}^{\ast}B^{\dagger}EA^{\dagger}U=\begin{pmatrix}
\widetilde{V}_{1}^{\ast}V_{1}\Sigma_{1}^{-1}-\widetilde{\Sigma}_{1}^{-1}\widetilde{U}_{1}^{\ast}U_{1} & 0 \\
0 & 0
\end{pmatrix},\\
&V^{\ast}A^{\dagger}EB^{\dagger}\widetilde{U}=\begin{pmatrix}
\Sigma_{1}^{-1}U_{1}^{\ast}\widetilde{U}_{1}-V_{1}^{\ast}\widetilde{V}_{1}\widetilde{\Sigma}_{1}^{-1} & 0 \\
0 & 0
\end{pmatrix}.
\end{align*}
Hence,
\begin{subequations}
\begin{align}
&\|B^{\dagger}EA^{\dagger}\|_{F}^{2}=\|\widetilde{V}_{1}^{\ast}V_{1}\Sigma_{1}^{-1}-\widetilde{\Sigma}_{1}^{-1}\widetilde{U}_{1}^{\ast}U_{1}\|_{F}^{2},\label{B+EA+}\\
&\|A^{\dagger}EB^{\dagger}\|_{F}^{2}=\|\Sigma_{1}^{-1}U_{1}^{\ast}\widetilde{U}_{1}-V_{1}^{\ast}\widetilde{V}_{1}\widetilde{\Sigma}_{1}^{-1}\|_{F}^{2}.\label{A+EB+}
\end{align}
\end{subequations}

Combining~\eqref{rela1.1} and~\eqref{B+EA+}, we can arrive at the identity~\eqref{ide1.1}. Similarly, the identity~\eqref{ide1.2} follows immediately from~\eqref{rela1.2} and~\eqref{A+EB+}.
\end{proof}

The following corollary can be directly deduced from Lemma~\ref{identity}.

\begin{corollary}
Under the assumptions of Lemma~{\rm\ref{identity}}, we have
\begin{subequations}
\begin{align}
&\|B^{\dagger}-A^{\dagger}\|_{F}^{2}\leq\|B^{\dagger}\|_{2}^{2}\|\widetilde{U}_{1}^{\ast}U_{2}\|_{F}^{2}+\|A^{\dagger}\|_{2}^{2}\|\widetilde{V}_{2}^{\ast}V_{1}\|_{F}^{2}+\|B^{\dagger}EA^{\dagger}\|_{F}^{2},\label{upE+1.1}\\
&\|B^{\dagger}-A^{\dagger}\|_{F}^{2}\leq\|A^{\dagger}\|_{2}^{2}\|\widetilde{U}_{2}^{\ast}U_{1}\|_{F}^{2}+\|B^{\dagger}\|_{2}^{2}\|\widetilde{V}_{1}^{\ast}V_{2}\|_{F}^{2}+\|A^{\dagger}EB^{\dagger}\|_{F}^{2}.\label{upE+1.2}
\end{align}
\end{subequations}
\end{corollary}

If $\rank(B)=\rank(A)$, then $\|\widetilde{U}_{1}^{\ast}U_{2}\|_{F}$, $\|\widetilde{U}_{2}^{\ast}U_{1}\|_{F}$, $\|\widetilde{V}_{1}^{\ast}V_{2}\|_{F}$, and $\|\widetilde{V}_{2}^{\ast}V_{1}\|_{F}$ have the following relations~\cite[Lemma~2.2]{Chen2016}.

\begin{lemma}\label{s=r}
Let $A\in\mathbb{C}^{m\times n}_{r}$ and $B\in\mathbb{C}^{m\times n}_{s}$ have the SVDs~\eqref{SVD-A} and~\eqref{SVD-B}, respectively. If $s=r$, then
\begin{displaymath}
\|\widetilde{U}_{1}^{\ast}U_{2}\|_{F}=\|\widetilde{U}_{2}^{\ast}U_{1}\|_{F} \quad \text{and} \quad \|\widetilde{V}_{1}^{\ast}V_{2}\|_{F}=\|\widetilde{V}_{2}^{\ast}V_{1}\|_{F}.
\end{displaymath}
\end{lemma}

\section{Main results}

\label{sec:main}

\setcounter{equation}{0}

In this section, we develop some new perturbation bounds for $\|B^{\dagger}-A^{\dagger}\|_{F}^{2}$. The first estimate depends only on the positive singular values of $A$ and $B$.

\begin{theorem}
Let $A\in\mathbb{C}^{m\times n}_{r}$ and $B\in\mathbb{C}^{m\times n}_{s}$ have the positive singular values $\{\sigma_{i}\}_{i=1}^{r}$ and $\{\widetilde{\sigma}_{i}\}_{i=1}^{s}$, respectively, where $\sigma_{1}\geq\cdots\geq\sigma_{r}$ and $\widetilde{\sigma}_{1}\geq\cdots\geq\widetilde{\sigma}_{s}$.

{\rm (i)} If $s\leq r$, then
\begin{equation}\label{up+low1}
\sum_{i=1}^{s}\bigg(\frac{1}{\sigma_{i}}-\frac{1}{\widetilde{\sigma}_{i}}\bigg)^{2}+\sum_{i=s+1}^{r}\frac{1}{\sigma_{i}^{2}}\leq\|B^{\dagger}-A^{\dagger}\|_{F}^{2}\leq\sum_{i=1}^{s}\bigg(\frac{1}{\sigma_{i}}+\frac{1}{\widetilde{\sigma}_{i}}\bigg)^{2}+\sum_{i=s+1}^{r}\frac{1}{\sigma_{i}^{2}}.
\end{equation}

{\rm (ii)} If $s>r$, then
\begin{equation}\label{up+low2}
\sum_{i=1}^{r}\bigg(\frac{1}{\sigma_{i}}-\frac{1}{\widetilde{\sigma}_{i}}\bigg)^{2}+\sum_{i=r+1}^{s}\frac{1}{\widetilde{\sigma}_{i}^{2}}\leq\|B^{\dagger}-A^{\dagger}\|_{F}^{2}\leq\sum_{i=1}^{r}\bigg(\frac{1}{\sigma_{i}}+\frac{1}{\widetilde{\sigma}_{i}}\bigg)^{2}+\sum_{i=r+1}^{s}\frac{1}{\widetilde{\sigma}_{i}^{2}}.
\end{equation}
\end{theorem}

\begin{proof}
Observe first that
\begin{equation}\label{trace}
\|B^{\dagger}-A^{\dagger}\|_{F}^{2}=\tr\big((B^{\dagger}-A^{\dagger})^{\ast}(B^{\dagger}-A^{\dagger})\big)=\sum_{i=1}^{r}\frac{1}{\sigma_{i}^{2}}+\sum_{i=1}^{s}\frac{1}{\widetilde{\sigma}_{i}^{2}}-2\Retr\big(A^{\dagger}(B^{\dagger})^{\ast}\big).
\end{equation}
From~\eqref{von}, we deduce that
\begin{equation}\label{real}
-\sum_{i=1}^{\min\{s,r\}}\frac{1}{\sigma_{i}\widetilde{\sigma}_{i}}\leq\Retr\big(A^{\dagger}(B^{\dagger})^{\ast}\big)\leq\sum_{i=1}^{\min\{s,r\}}\frac{1}{\sigma_{i}\widetilde{\sigma}_{i}}.
\end{equation}
Combining~\eqref{trace} and~\eqref{real}, we can obtain the inequalities~\eqref{up+low1} and~\eqref{up+low2}.
\end{proof}

\begin{remark}\rm
According to the lower bounds in~\eqref{up+low1} and~\eqref{up+low2}, we conclude that a necessary condition for $\lim\limits_{B\rightarrow A}B^{\dagger}=A^{\dagger}$ ($B$ is viewed as a variable) is that $\rank(B)=\rank(A)$ always holds as $B$ tends to $A$. In fact, it is also a sufficient condition for $\lim\limits_{B\rightarrow A}B^{\dagger}=A^{\dagger}$~\cite{Stewart1969}.
\end{remark}

\begin{remark}\rm
Under the setting of~\eqref{Ex0}, the lower and upper bounds in~\eqref{up+low2} are $4\tau^{2}+\frac{1}{\tau^{2}}$ and $4(1+\tau)^{2}+\frac{1}{\tau^{2}}$, respectively. Clearly, the lower bound has attained $\|B^{\dagger}-A^{\dagger}\|_{F}^{2}$, and the upper bound will be very tight when $\tau$ is small.
\end{remark}

In what follows, we establish some upper bounds involving the perturbation $E=B-A$.

\begin{theorem}\label{up-thm1}
Let $A\in\mathbb{C}^{m\times n}_{r}$, $B\in\mathbb{C}^{m\times n}_{s}$, and $E=B-A$. Then
\begin{equation}\label{up1}
\|B^{\dagger}-A^{\dagger}\|_{F}^{2}\leq\min\big\{\alpha_{1}+\|B^{\dagger}EA^{\dagger}\|_{F}^{2},\alpha_{2}+\|A^{\dagger}EB^{\dagger}\|_{F}^{2}\big\},
\end{equation}
where
\begin{align*}
&\alpha_{1}:=\|A^{\dagger}\|_{2}^{2}\big(\|A^{\dagger}E\|_{F}^{2}-\|A^{\dagger}EB^{\dagger}B\|_{F}^{2}\big)+\|B^{\dagger}\|_{2}^{2}\big(\|EB^{\dagger}\|_{F}^{2}-\|AA^{\dagger}EB^{\dagger}\|_{F}^{2}\big),\\
&\alpha_{2}:=\|A^{\dagger}\|_{2}^{2}\big(\|EA^{\dagger}\|_{F}^{2}-\|BB^{\dagger}EA^{\dagger}\|_{F}^{2}\big)+\|B^{\dagger}\|_{2}^{2}\big(\|B^{\dagger}E\|_{F}^{2}-\|B^{\dagger}EA^{\dagger}A\|_{F}^{2}\big).
\end{align*}
\end{theorem}

\begin{proof}
By~\eqref{SVD-A}, \eqref{SVD-B}, \eqref{MP-A}, and~\eqref{MP-B}, we have
\begin{align*}
&U^{\ast}EB^{\dagger}\widetilde{U}=\begin{pmatrix}
U_{1}^{\ast}\widetilde{U}_{1}-\Sigma_{1}V_{1}^{\ast}\widetilde{V}_{1}\widetilde{\Sigma}_{1}^{-1} & 0 \\
U_{2}^{\ast}\widetilde{U}_{1} & 0
\end{pmatrix},\\
&U^{\ast}AA^{\dagger}EB^{\dagger}\widetilde{U}=\begin{pmatrix}
U_{1}^{\ast}\widetilde{U}_{1}-\Sigma_{1}V_{1}^{\ast}\widetilde{V}_{1}\widetilde{\Sigma}_{1}^{-1} & 0 \\
0 & 0
\end{pmatrix}.
\end{align*}
Hence,
\begin{equation}\label{U1starU2}
\|\widetilde{U}_{1}^{\ast}U_{2}\|_{F}^{2}=\|EB^{\dagger}\|_{F}^{2}-\|AA^{\dagger}EB^{\dagger}\|_{F}^{2}.
\end{equation}
Similarly, we have
\begin{align*}
&V^{\ast}A^{\dagger}E\widetilde{V}=\begin{pmatrix}
\Sigma_{1}^{-1}U_{1}^{\ast}\widetilde{U}_{1}\widetilde{\Sigma}_{1}-V_{1}^{\ast}\widetilde{V}_{1} & -V_{1}^{\ast}\widetilde{V}_{2} \\
0 & 0
\end{pmatrix},\\
&V^{\ast}A^{\dagger}EB^{\dagger}B\widetilde{V}=\begin{pmatrix}
\Sigma_{1}^{-1}U_{1}^{\ast}\widetilde{U}_{1}\widetilde{\Sigma}_{1}-V_{1}^{\ast}\widetilde{V}_{1} & 0 \\
0 & 0
\end{pmatrix}.
\end{align*}
Thus,
\begin{equation}\label{V2starV1}
\|\widetilde{V}_{2}^{\ast}V_{1}\|_{F}^{2}=\|A^{\dagger}E\|_{F}^{2}-\|A^{\dagger}EB^{\dagger}B\|_{F}^{2}.
\end{equation}
Using~\eqref{upE+1.1}, \eqref{U1starU2}, and~\eqref{V2starV1}, we obtain
\begin{equation}\label{part1.1}
\|B^{\dagger}-A^{\dagger}\|_{F}^{2}\leq\alpha_{1}+\|B^{\dagger}EA^{\dagger}\|_{F}^{2}.
\end{equation}
Interchanging the roles of $A$ and $B$ yields
\begin{equation}\label{part1.2}
\|B^{\dagger}-A^{\dagger}\|_{F}^{2}\leq\alpha_{2}+\|A^{\dagger}EB^{\dagger}\|_{F}^{2}.
\end{equation}
The desired result~\eqref{up1} follows immediately by combining~\eqref{part1.1} and~\eqref{part1.2}.
\end{proof}

\begin{remark}\rm
If $A\in\mathbb{C}^{m\times n}_{n}$ $(m\geq n)$, then $V_{1}=V$ and $V_{2}$ vanishes. In this case, we have
\begin{align*}
&\alpha_{1}=\|B^{\dagger}\|_{2}^{2}\big(\|EB^{\dagger}\|_{F}^{2}-\|AA^{\dagger}EB^{\dagger}\|_{F}^{2}\big)+(n-s)\|A^{\dagger}\|_{2}^{2},\\
&\alpha_{2}=\|A^{\dagger}\|_{2}^{2}\big(\|EA^{\dagger}\|_{F}^{2}-\|BB^{\dagger}EA^{\dagger}\|_{F}^{2}\big).
\end{align*}
Note that
\begin{displaymath}
\|AA^{\dagger}EB^{\dagger}\|_{F}^{2}\geq\frac{\|A^{\dagger}EB^{\dagger}\|_{F}^{2}}{\|A^{\dagger}\|_{2}^{2}} \quad \text{and} \quad \|BB^{\dagger}EA^{\dagger}\|_{F}^{2}\geq\frac{\|B^{\dagger}EA^{\dagger}\|_{F}^{2}}{\|B^{\dagger}\|_{2}^{2}}.
\end{displaymath}
In light of~\eqref{up1}, we have
\begin{align*}
\|B^{\dagger}-A^{\dagger}\|_{F}^{2}&\leq\frac{\|A^{\dagger}\|_{2}^{2}\|B^{\dagger}\|_{2}^{2}}{\|A^{\dagger}\|_{2}^{2}+\|B^{\dagger}\|_{2}^{2}}\bigg(\frac{\alpha_{1}+\|B^{\dagger}EA^{\dagger}\|_{F}^{2}}{\|B^{\dagger}\|_{2}^{2}}+\frac{\alpha_{2}+\|A^{\dagger}EB^{\dagger}\|_{F}^{2}}{\|A^{\dagger}\|_{2}^{2}}\bigg)\\
&\leq\frac{\|A^{\dagger}\|_{2}^{2}\|B^{\dagger}\|_{2}^{2}}{\|A^{\dagger}\|_{2}^{2}+\|B^{\dagger}\|_{2}^{2}}\bigg(\|EA^{\dagger}\|_{F}^{2}+\|EB^{\dagger}\|_{F}^{2}+(n-s)\frac{\|A^{\dagger}\|_{2}^{2}}{\|B^{\dagger}\|_{2}^{2}}\bigg).
\end{align*}
Therefore, \eqref{up1} has improved the estimate~\eqref{Li2.1}.
\end{remark}

\begin{remark}\rm
If $A\in\mathbb{C}^{m\times n}_{n}$ and $B\in\mathbb{C}^{m\times n}_{n}$ $(m\geq n)$, then
\begin{align*}
&\alpha_{1}=\|B^{\dagger}\|_{2}^{2}\big(\|EB^{\dagger}\|_{F}^{2}-\|AA^{\dagger}EB^{\dagger}\|_{F}^{2}\big),\\ &\alpha_{2}=\|A^{\dagger}\|_{2}^{2}\big(\|EA^{\dagger}\|_{F}^{2}-\|BB^{\dagger}EA^{\dagger}\|_{F}^{2}\big).
\end{align*}
Due to $\rank(B)=\rank(A)$, it follows that
\begin{displaymath}
\|EB^{\dagger}\|_{F}^{2}-\|AA^{\dagger}EB^{\dagger}\|_{F}^{2}=\|\widetilde{U}_{1}^{\ast}U_{2}\|_{F}^{2}=\|\widetilde{U}_{2}^{\ast}U_{1}\|_{F}^{2}=\|EA^{\dagger}\|_{F}^{2}-\|BB^{\dagger}EA^{\dagger}\|_{F}^{2},
\end{displaymath}
where we have used Lemma~\ref{s=r}. Then
\begin{align*}
&\alpha_{1}=\|B^{\dagger}\|_{2}^{2}\big(\|EA^{\dagger}\|_{F}^{2}-\|BB^{\dagger}EA^{\dagger}\|_{F}^{2}\big)\leq\|B^{\dagger}\|_{2}^{2}\|EA^{\dagger}\|_{F}^{2}-\|B^{\dagger}EA^{\dagger}\|_{F}^{2},\\
&\alpha_{2}=\|A^{\dagger}\|_{2}^{2}\big(\|EB^{\dagger}\|_{F}^{2}-\|AA^{\dagger}EB^{\dagger}\|_{F}^{2}\big)\leq\|A^{\dagger}\|_{2}^{2}\|EB^{\dagger}\|_{F}^{2}-\|A^{\dagger}EB^{\dagger}\|_{F}^{2},
\end{align*}
because
\begin{displaymath}
\|BB^{\dagger}EA^{\dagger}\|_{F}^{2}\geq\frac{\|B^{\dagger}EA^{\dagger}\|_{F}^{2}}{\|B^{\dagger}\|_{2}^{2}} \quad \text{and} \quad \|AA^{\dagger}EB^{\dagger}\|_{F}^{2}\geq\frac{\|A^{\dagger}EB^{\dagger}\|_{F}^{2}}{\|A^{\dagger}\|_{2}^{2}}.
\end{displaymath}
By~\eqref{up1}, we have
\begin{displaymath}
\|B^{\dagger}-A^{\dagger}\|_{F}^{2}\leq\min\big\{\|B^{\dagger}\|_{2}^{2}\|EA^{\dagger}\|_{F}^{2},\|A^{\dagger}\|_{2}^{2}\|EB^{\dagger}\|_{F}^{2}\big\},
\end{displaymath}
which is exactly~\eqref{Li2.2}. Thus, \eqref{up1} has also improved the estimate~\eqref{Li2.2}.
\end{remark}

The following two corollaries are based on Theorem~\ref{up-thm1}.

\begin{corollary}\label{corup1}
Under the assumptions of Theorem~{\rm\ref{up-thm1}}, it holds that
\begin{equation}\label{corup1.1}
\|B^{\dagger}-A^{\dagger}\|_{F}^{2}\leq\min\big\{\beta_{1}+\|B^{\dagger}EA^{\dagger}\|_{F}^{2},\beta_{2}+\|A^{\dagger}EB^{\dagger}\|_{F}^{2}\big\},
\end{equation}
where
\begin{align*}
&\beta_{1}:=\|A^{\dagger}\|_{2}^{4}\big(\|E\|_{F}^{2}-\|EB^{\dagger}B\|_{F}^{2}\big)+\|B^{\dagger}\|_{2}^{4}\big(\|E\|_{F}^{2}-\|AA^{\dagger}E\|_{F}^{2}\big),\\
&\beta_{2}:=\|A^{\dagger}\|_{2}^{4}\big(\|E\|_{F}^{2}-\|BB^{\dagger}E\|_{F}^{2}\big)+\|B^{\dagger}\|_{2}^{4}\big(\|E\|_{F}^{2}-\|EA^{\dagger}A\|_{F}^{2}\big).
\end{align*}
\end{corollary}

\begin{proof}
In view of~\eqref{SVD-A}, \eqref{SVD-B}, \eqref{MP-A}, and~\eqref{MP-B}, we have
\begin{align}
&U^{\ast}E\widetilde{V}=\begin{pmatrix}
U_{1}^{\ast}\widetilde{U}_{1}\widetilde{\Sigma}_{1}-\Sigma_{1}V_{1}^{\ast}\widetilde{V}_{1} & -\Sigma_{1}V_{1}^{\ast}\widetilde{V}_{2} \\
U_{2}^{\ast}\widetilde{U}_{1}\widetilde{\Sigma}_{1} & 0
\end{pmatrix},\label{E1}\\
&U^{\ast}EB^{\dagger}B\widetilde{V}=\begin{pmatrix}
U_{1}^{\ast}\widetilde{U}_{1}\widetilde{\Sigma}_{1}-\Sigma_{1}V_{1}^{\ast}\widetilde{V}_{1} & 0 \\
U_{2}^{\ast}\widetilde{U}_{1}\widetilde{\Sigma}_{1} & 0
\end{pmatrix},\label{EB+B}\\
&U^{\ast}AA^{\dagger}E\widetilde{V}=\begin{pmatrix}
U_{1}^{\ast}\widetilde{U}_{1}\widetilde{\Sigma}_{1}-\Sigma_{1}V_{1}^{\ast}\widetilde{V}_{1} & -\Sigma_{1}V_{1}^{\ast}\widetilde{V}_{2} \\
0 & 0
\end{pmatrix}.\label{AA+E}
\end{align}
From~\eqref{E1}, \eqref{EB+B}, and~\eqref{AA+E}, we deduce that
\begin{align*}
&\|\widetilde{V}_{2}^{\ast}V_{1}\Sigma_{1}\|_{F}^{2}=\|E\|_{F}^{2}-\|EB^{\dagger}B\|_{F}^{2},\\ &\|\widetilde{\Sigma}_{1}\widetilde{U}_{1}^{\ast}U_{2}\|_{F}^{2}=\|E\|_{F}^{2}-\|AA^{\dagger}E\|_{F}^{2}.
\end{align*}
Then
\begin{align*}
&\|A^{\dagger}E\|_{F}^{2}-\|A^{\dagger}EB^{\dagger}B\|_{F}^{2}=\|\widetilde{V}_{2}^{\ast}V_{1}\|_{F}^{2}\leq\|A^{\dagger}\|_{2}^{2}\|\widetilde{V}_{2}^{\ast}V_{1}\Sigma_{1}\|_{F}^{2}=\|A^{\dagger}\|_{2}^{2}\big(\|E\|_{F}^{2}-\|EB^{\dagger}B\|_{F}^{2}\big),\\
&\|EB^{\dagger}\|_{F}^{2}-\|AA^{\dagger}EB^{\dagger}\|_{F}^{2}=\|\widetilde{U}_{1}^{\ast}U_{2}\|_{F}^{2}\leq\|B^{\dagger}\|_{2}^{2}\|\widetilde{\Sigma}_{1}\widetilde{U}_{1}^{\ast}U_{2}\|_{F}^{2}=\|B^{\dagger}\|_{2}^{2}\big(\|E\|_{F}^{2}-\|AA^{\dagger}E\|_{F}^{2}\big).
\end{align*}
Hence,
\begin{displaymath}
\alpha_{1}\leq\|A^{\dagger}\|_{2}^{4}\big(\|E\|_{F}^{2}-\|EB^{\dagger}B\|_{F}^{2}\big)+\|B^{\dagger}\|_{2}^{4}\big(\|E\|_{F}^{2}-\|AA^{\dagger}E\|_{F}^{2}\big).
\end{displaymath}
Similarly,
\begin{displaymath}
\alpha_{2}\leq\|A^{\dagger}\|_{2}^{4}\big(\|E\|_{F}^{2}-\|BB^{\dagger}E\|_{F}^{2}\big)+\|B^{\dagger}\|_{2}^{4}\big(\|E\|_{F}^{2}-\|EA^{\dagger}A\|_{F}^{2}\big).
\end{displaymath}
Using~\eqref{up1}, we can obtain the estimate~\eqref{corup1.1} immediately.
\end{proof}

\begin{corollary}\label{corup2}
Under the assumptions of Theorem~{\rm\ref{up-thm1}}, it holds that
\begin{equation}\label{corup1.2}
\|B^{\dagger}-A^{\dagger}\|_{F}^{2}\leq\min\big\{\gamma_{1}+\|B^{\dagger}EA^{\dagger}\|_{F}^{2}, \gamma_{2}+\|A^{\dagger}EB^{\dagger}\|_{F}^{2}\big\},
\end{equation}
where
\begin{align*}
&\gamma_{1}:=\|A^{\dagger}\|_{2}^{2}\bigg(\|A^{\dagger}E\|_{F}^{2}-\frac{\|A^{\dagger}EB^{\dagger}\|_{F}^{2}}{\|B^{\dagger}\|_{2}^{2}}\bigg)+\|B^{\dagger}\|_{2}^{2}\bigg(\|EB^{\dagger}\|_{F}^{2}-\frac{\|A^{\dagger}EB^{\dagger}\|_{F}^{2}}{\|A^{\dagger}\|_{2}^{2}}\bigg),\\
&\gamma_{2}:=\|A^{\dagger}\|_{2}^{2}\bigg(\|EA^{\dagger}\|_{F}^{2}-\frac{\|B^{\dagger}EA^{\dagger}\|_{F}^{2}}{\|B^{\dagger}\|_{2}^{2}}\bigg)+\|B^{\dagger}\|_{2}^{2}\bigg(\|B^{\dagger}E\|_{F}^{2}-\frac{\|B^{\dagger}EA^{\dagger}\|_{F}^{2}}{\|A^{\dagger}\|_{2}^{2}}\bigg).
\end{align*}
\end{corollary}

\begin{proof}
It is easy to verify that
\begin{displaymath}
\|A^{\dagger}EB^{\dagger}B\|_{F}^{2}\geq\frac{\|A^{\dagger}EB^{\dagger}\|_{F}^{2}}{\|B^{\dagger}\|_{2}^{2}} \quad \text{and} \quad \|AA^{\dagger}EB^{\dagger}\|_{F}^{2}\geq\frac{\|A^{\dagger}EB^{\dagger}\|_{F}^{2}}{\|A^{\dagger}\|_{2}^{2}}.
\end{displaymath}
The estimate~\eqref{corup1.2} then follows from Theorem~\ref{up-thm1}.
\end{proof}

The following theorem provides the sharper counterparts of~\eqref{Meng1}, \eqref{Meng2}, and~\eqref{Li1}.

\begin{theorem}\label{up-thm2}
Let $A\in\mathbb{C}^{m\times n}_{r}$, $B\in\mathbb{C}^{m\times n}_{s}$, and $E=B-A$. Then
\begin{equation}\label{up2.1}
\|B^{\dagger}-A^{\dagger}\|_{F}^{2}\leq\min\big\{\delta_{1}+\|B^{\dagger}EA^{\dagger}\|_{F}^{2}, \delta_{2}+\|A^{\dagger}EB^{\dagger}\|_{F}^{2}\big\},
\end{equation}
where
\begin{align*}
&\delta_{1}:=\max\big\{\|A^{\dagger}\|_{2}^{4},\|B^{\dagger}\|_{2}^{4}\big\}\bigg(\|E\|_{F}^{2}-\max\bigg\{\frac{\|AA^{\dagger}EB^{\dagger}\|_{F}^{2}}{\|B^{\dagger}\|_{2}^{2}},\frac{\|A^{\dagger}EB^{\dagger}B\|_{F}^{2}}{\|A^{\dagger}\|_{2}^{2}}\bigg\}\bigg),\\
&\delta_{2}:=\max\big\{\|A^{\dagger}\|_{2}^{4},\|B^{\dagger}\|_{2}^{4}\big\}\bigg(\|E\|_{F}^{2}-\max\bigg\{\frac{\|BB^{\dagger}EA^{\dagger}\|_{F}^{2}}{\|A^{\dagger}\|_{2}^{2}},\frac{\|B^{\dagger}EA^{\dagger}A\|_{F}^{2}}{\|B^{\dagger}\|_{2}^{2}}\bigg\}\bigg).
\end{align*}
In particular, if $s=r$, then
\begin{equation}\label{up2.2}
\|B^{\dagger}-A^{\dagger}\|_{F}^{2}\leq\min\big\{\varepsilon_{1}+\|B^{\dagger}EA^{\dagger}\|_{F}^{2}, \varepsilon_{2}+\|A^{\dagger}EB^{\dagger}\|_{F}^{2}\big\},
\end{equation}
where
\begin{align*}
&\varepsilon_{1}:=\|A^{\dagger}\|_{2}^{2}\|B^{\dagger}\|_{2}^{2}\bigg(\|E\|_{F}^{2}-\max\bigg\{\frac{\|BB^{\dagger}EA^{\dagger}\|_{F}^{2}}{\|A^{\dagger}\|_{2}^{2}},\frac{\|B^{\dagger}EA^{\dagger}A\|_{F}^{2}}{\|B^{\dagger}\|_{2}^{2}}\bigg\}\bigg),\\
&\varepsilon_{2}:=\|A^{\dagger}\|_{2}^{2}\|B^{\dagger}\|_{2}^{2}\bigg(\|E\|_{F}^{2}-\max\bigg\{\frac{\|AA^{\dagger}EB^{\dagger}\|_{F}^{2}}{\|B^{\dagger}\|_{2}^{2}},\frac{\|A^{\dagger}EB^{\dagger}B\|_{F}^{2}}{\|A^{\dagger}\|_{2}^{2}}\bigg\}\bigg).
\end{align*}
\end{theorem}

\begin{proof}
According to~\eqref{E1}, we deduce that
\begin{align*}
\|E\|_{F}^{2}&=\|(U_{1}^{\ast}\widetilde{U}_{1}-\Sigma_{1}V_{1}^{\ast}\widetilde{V}_{1}\widetilde{\Sigma}_{1}^{-1})\widetilde{\Sigma}_{1}\|_{F}^{2}+\|\Sigma_{1}V_{1}^{\ast}\widetilde{V}_{2}\|_{F}^{2}+\|U_{2}^{\ast}\widetilde{U}_{1}\widetilde{\Sigma}_{1}\|_{F}^{2}\\
&\geq\frac{\|AA^{\dagger}EB^{\dagger}\|_{F}^{2}}{\|B^{\dagger}\|_{2}^{2}}+\frac{\|\widetilde{V}_{2}^{\ast}V_{1}\|_{F}^{2}}{\|A^{\dagger}\|_{2}^{2}}+\frac{\|\widetilde{U}_{1}^{\ast}U_{2}\|_{F}^{2}}{\|B^{\dagger}\|_{2}^{2}}
\end{align*}
and
\begin{align*}
\|E\|_{F}^{2}&=\|\Sigma_{1}(\Sigma_{1}^{-1}U_{1}^{\ast}\widetilde{U}_{1}\widetilde{\Sigma}_{1}-V_{1}^{\ast}\widetilde{V}_{1})\|_{F}^{2}+\|\Sigma_{1}V_{1}^{\ast}\widetilde{V}_{2}\|_{F}^{2}+\|U_{2}^{\ast}\widetilde{U}_{1}\widetilde{\Sigma}_{1}\|_{F}^{2}\\
&\geq\frac{\|A^{\dagger}EB^{\dagger}B\|_{F}^{2}}{\|A^{\dagger}\|_{2}^{2}}+\frac{\|\widetilde{V}_{2}^{\ast}V_{1}\|_{F}^{2}}{\|A^{\dagger}\|_{2}^{2}}+\frac{\|\widetilde{U}_{1}^{\ast}U_{2}\|_{F}^{2}}{\|B^{\dagger}\|_{2}^{2}}.
\end{align*}
Hence,
\begin{equation}\label{up-UV1}
\frac{\|\widetilde{U}_{1}^{\ast}U_{2}\|_{F}^{2}}{\|B^{\dagger}\|_{2}^{2}}+\frac{\|\widetilde{V}_{2}^{\ast}V_{1}\|_{F}^{2}}{\|A^{\dagger}\|_{2}^{2}}\leq\|E\|_{F}^{2}-\max\bigg\{\frac{\|AA^{\dagger}EB^{\dagger}\|_{F}^{2}}{\|B^{\dagger}\|_{2}^{2}},\frac{\|A^{\dagger}EB^{\dagger}B\|_{F}^{2}}{\|A^{\dagger}\|_{2}^{2}}\bigg\}.
\end{equation}
Using~\eqref{upE+1.1} and~\eqref{up-UV1}, we obtain
\begin{align*}
\|B^{\dagger}-A^{\dagger}\|_{F}^{2}&\leq\max\big\{\|A^{\dagger}\|_{2}^{4},\|B^{\dagger}\|_{2}^{4}\big\}\bigg(\frac{\|\widetilde{U}_{1}^{\ast}U_{2}\|_{F}^{2}}{\|B^{\dagger}\|_{2}^{2}}+\frac{\|\widetilde{V}_{2}^{\ast}V_{1}\|_{F}^{2}}{\|A^{\dagger}\|_{2}^{2}}\bigg)+\|B^{\dagger}EA^{\dagger}\|_{F}^{2}\\
&\leq\delta_{1}+\|B^{\dagger}EA^{\dagger}\|_{F}^{2}.
\end{align*}
Interchanging the roles of $A$ and $B$, we arrive at
\begin{displaymath}
\|B^{\dagger}-A^{\dagger}\|_{F}^{2}\leq\delta_{2}+\|A^{\dagger}EB^{\dagger}\|_{F}^{2}.
\end{displaymath}
Thus, the estimate~\eqref{up2.1} is valid.

We next consider the special case $s=r$. Direct computation yields
\begin{displaymath}
\widetilde{U}^{\ast}EV=\begin{pmatrix}
\widetilde{\Sigma}_{1}\widetilde{V}_{1}^{\ast}V_{1}-\widetilde{U}_{1}^{\ast}U_{1}\Sigma_{1} & \widetilde{\Sigma}_{1}\widetilde{V}_{1}^{\ast}V_{2} \\
-\widetilde{U}_{2}^{\ast}U_{1}\Sigma_{1} & 0
\end{pmatrix},
\end{displaymath}
which leads to
\begin{equation}\label{E2}
\|E\|_{F}^{2}=\|\widetilde{\Sigma}_{1}\widetilde{V}_{1}^{\ast}V_{1}-\widetilde{U}_{1}^{\ast}U_{1}\Sigma_{1}\|_{F}^{2}+\|\widetilde{\Sigma}_{1}\widetilde{V}_{1}^{\ast}V_{2}\|_{F}^{2}+\|\widetilde{U}_{2}^{\ast}U_{1}\Sigma_{1}\|_{F}^{2}.
\end{equation}
If $s=r$, we get from~\eqref{E2} and Lemma~\ref{s=r} that
\begin{align*}
\|E\|_{F}^{2}&=\|(\widetilde{\Sigma}_{1}\widetilde{V}_{1}^{\ast}V_{1}\Sigma_{1}^{-1}-\widetilde{U}_{1}^{\ast}U_{1})\Sigma_{1}\|_{F}^{2}+\|\widetilde{\Sigma}_{1}\widetilde{V}_{1}^{\ast}V_{2}\|_{F}^{2}+\|\widetilde{U}_{2}^{\ast}U_{1}\Sigma_{1}\|_{F}^{2}\\
&\geq\frac{\|BB^{\dagger}EA^{\dagger}\|_{F}^{2}}{\|A^{\dagger}\|_{2}^{2}}+\frac{\|\widetilde{V}_{2}^{\ast}V_{1}\|_{F}^{2}}{\|B^{\dagger}\|_{2}^{2}}+\frac{\|\widetilde{U}_{1}^{\ast}U_{2}\|_{F}^{2}}{\|A^{\dagger}\|_{2}^{2}}
\end{align*}
and
\begin{align*}
\|E\|_{F}^{2}&=\|\widetilde{\Sigma}_{1}(\widetilde{V}_{1}^{\ast}V_{1}-\widetilde{\Sigma}_{1}^{-1}\widetilde{U}_{1}^{\ast}U_{1}\Sigma_{1})\|_{F}^{2}+\|\widetilde{\Sigma}_{1}\widetilde{V}_{1}^{\ast}V_{2}\|_{F}^{2}+\|\widetilde{U}_{2}^{\ast}U_{1}\Sigma_{1}\|_{F}^{2}\\
&\geq\frac{\|B^{\dagger}EA^{\dagger}A\|_{F}^{2}}{\|B^{\dagger}\|_{2}^{2}}+\frac{\|\widetilde{V}_{2}^{\ast}V_{1}\|_{F}^{2}}{\|B^{\dagger}\|_{2}^{2}}+\frac{\|\widetilde{U}_{1}^{\ast}U_{2}\|_{F}^{2}}{\|A^{\dagger}\|_{2}^{2}}.
\end{align*}
Hence,
\begin{equation}\label{up-UV2}
\frac{\|\widetilde{U}_{1}^{\ast}U_{2}\|_{F}^{2}}{\|A^{\dagger}\|_{2}^{2}}+\frac{\|\widetilde{V}_{2}^{\ast}V_{1}\|_{F}^{2}}{\|B^{\dagger}\|_{2}^{2}}\leq\|E\|_{F}^{2}-\max\bigg\{\frac{\|BB^{\dagger}EA^{\dagger}\|_{F}^{2}}{\|A^{\dagger}\|_{2}^{2}},\frac{\|B^{\dagger}EA^{\dagger}A\|_{F}^{2}}{\|B^{\dagger}\|_{2}^{2}}\bigg\}.
\end{equation}
By~\eqref{upE+1.1} and~\eqref{up-UV2}, we have
\begin{align*}
\|B^{\dagger}-A^{\dagger}\|_{F}^{2}&\leq\|A^{\dagger}\|_{2}^{2}\|B^{\dagger}\|_{2}^{2}\bigg(\frac{\|\widetilde{U}_{1}^{\ast}U_{2}\|_{F}^{2}}{\|A^{\dagger}\|_{2}^{2}}+\frac{\|\widetilde{V}_{2}^{\ast}V_{1}\|_{F}^{2}}{\|B^{\dagger}\|_{2}^{2}}\bigg)+\|B^{\dagger}EA^{\dagger}\|_{F}^{2}\\
&\leq\varepsilon_{1}+\|B^{\dagger}EA^{\dagger}\|_{F}^{2}.
\end{align*}
Interchanging the roles of $A$ and $B$ yields
\begin{displaymath}
\|B^{\dagger}-A^{\dagger}\|_{F}^{2}\leq\varepsilon_{2}+\|A^{\dagger}EB^{\dagger}\|_{F}^{2}.
\end{displaymath}
Therefore, the estimate~\eqref{up2.2} is proved. This completes the proof.
\end{proof}

\begin{remark}\rm
From~\eqref{up2.1}, we deduce that
\begin{equation}\label{comp1}
\|B^{\dagger}-A^{\dagger}\|_{F}^{2}\leq\frac{1}{2}\big(\delta_{1}+\delta_{2}+\|A^{\dagger}EB^{\dagger}\|_{F}^{2}+\|B^{\dagger}EA^{\dagger}\|_{F}^{2}\big).
\end{equation}
Using the inequalities
\begin{displaymath}
\|AA^{\dagger}EB^{\dagger}\|_{F}^{2}\geq\frac{\|A^{\dagger}EB^{\dagger}\|_{F}^{2}}{\|A^{\dagger}\|_{2}^{2}} \quad \text{and} \quad \|BB^{\dagger}EA^{\dagger}\|_{F}^{2}\geq\frac{\|B^{\dagger}EA^{\dagger}\|_{F}^{2}}{\|B^{\dagger}\|_{2}^{2}},
\end{displaymath}
we obtain
\begin{align*}
\delta_{1}+\delta_{2}&\leq\max\big\{\|A^{\dagger}\|_{2}^{4},\|B^{\dagger}\|_{2}^{4}\big\}\bigg(2\|E\|_{F}^{2}-\frac{\|A^{\dagger}EB^{\dagger}\|_{F}^{2}+\|B^{\dagger}EA^{\dagger}\|_{F}^{2}}{\|A^{\dagger}\|_{2}^{2}\|B^{\dagger}\|_{2}^{2}}\bigg)\\
&=2\max\big\{\|A^{\dagger}\|_{2}^{4},\|B^{\dagger}\|_{2}^{4}\big\}\|E\|_{F}^{2}-\max\bigg\{\frac{\|A^{\dagger}\|_{2}^{2}}{\|B^{\dagger}\|_{2}^{2}},\frac{\|B^{\dagger}\|_{2}^{2}}{\|A^{\dagger}\|_{2}^{2}}\bigg\}\big(\|A^{\dagger}EB^{\dagger}\|_{F}^{2}+\|B^{\dagger}EA^{\dagger}\|_{F}^{2}\big)\\
&\leq 2\max\big\{\|A^{\dagger}\|_{2}^{4},\|B^{\dagger}\|_{2}^{4}\big\}\|E\|_{F}^{2}-\big(\|A^{\dagger}EB^{\dagger}\|_{F}^{2}+\|B^{\dagger}EA^{\dagger}\|_{F}^{2}\big).
\end{align*}
Thus, the estimate~\eqref{comp1} is sharper than both~\eqref{Meng1} and~\eqref{Li1}. Furthermore, since
\begin{displaymath}
\|B^{\dagger}\|_{2}^{2}\|BB^{\dagger}EA^{\dagger}\|_{F}^{2}\geq\|B^{\dagger}EA^{\dagger}\|_{F}^{2} \quad \text{and} \quad \|A^{\dagger}\|_{2}^{2}\|AA^{\dagger}EB^{\dagger}\|_{F}^{2}\geq\|A^{\dagger}EB^{\dagger}\|_{F}^{2},
\end{displaymath}
it follows that the estimate~\eqref{up2.2} is sharper than~\eqref{Meng2}.
\end{remark}

\begin{remark}\rm
Under the setting of~\eqref{Ex0}, the upper bounds in~\eqref{up1}, \eqref{corup1.1}, \eqref{corup1.2}, and~\eqref{up2.1} are listed in Table~\ref{tab:up}.
\begin{table}[h!!]
\centering
\setlength{\tabcolsep}{16mm}{
\begin{tabular}{@{} cc @{}}
\toprule
\text{Estimate} & \text{Upper bound for $\|B^{\dagger}-A^{\dagger}\|_{F}^{2}$} \\
\midrule
\eqref{up1} & $4\tau^{2}+\frac{1}{\tau^{2}}$ \\
\eqref{corup1.1} & $4\tau^{2}+\frac{1}{\tau^{2}}$ \\
\eqref{corup1.2} & $4\tau^{2}+\frac{1}{\tau^{2}}+\frac{4\tau^{2}}{(1+2\tau)^{2}}-4\tau^{4}$ \\
\eqref{up2.1} & $4\tau^{2}+\frac{1}{\tau^{2}}$ \\
\bottomrule
\end{tabular}}
\caption{\small The upper bounds in~\eqref{up1}, \eqref{corup1.1}, \eqref{corup1.2}, and~\eqref{up2.1}.}
\label{tab:up}
\end{table}
Table~\ref{tab:up} shows that the upper bounds in~\eqref{up1}, \eqref{corup1.1}, and~\eqref{up2.1} have attained the exact value $4\tau^{2}+\frac{1}{\tau^{2}}$. In addition, if $\tau$ is sufficiently small (i.e., the perturbed matrix is very close to the original one), the upper bound in~\eqref{corup1.2} will be very close to the exact value.
\end{remark}

\section*{Acknowledgments}

The author would like to thank the anonymous referees for their valuable comments and suggestions, which greatly improved the original version of this paper. This research was carried out by the author during his Ph.D. study at the Academy of Mathematics and Systems Science, Chinese Academy of Sciences. The author is grateful to Professor Chen-Song Zhang for his kind support.

\bibliographystyle{abbrv}
\bibliography{references}

\end{document}